\documentclass[12pt]{amsart}

\usepackage{amsmath}
\usepackage{amsfonts}
\usepackage{amssymb}
\usepackage{wasysym}
\usepackage{amsthm}
\usepackage[margin=1in]{geometry}
\usepackage{tikz}
\usepackage{hyperref}

\definecolor{my-linkcolor}{rgb}{0.75,0,0}
\definecolor{my-citecolor}{rgb}{0.1,0.57,0}
\definecolor{my-urlcolor}{rgb}{0,0,0.75}
\hypersetup{
	colorlinks, 
	linkcolor={my-linkcolor},
	citecolor={my-citecolor}, 
	urlcolor={my-urlcolor}
}

\title[Vertex algebraic intertwining operators]{Vertex algebraic intertwining operators among generalized Verma modules for affine Lie algebras}
\author{Robert McRae}
\date{}
\address{Yau Mathematical Sciences Center, Tsinghua University, Beijing 100084, China}
\curraddr{}
\email{rhmcrae@tsinghua.edu.cn}


    \theoremstyle{definition}\newtheorem{rema}{Remark}[section]
    \theoremstyle{plain}\newtheorem{propo}[rema]{Proposition}
    \newtheorem{theo}[rema]{Theorem}
    \theoremstyle{plain}
    \theoremstyle{definition}\newtheorem{defi}[rema]{Definition}
    \theoremstyle{plain}\newtheorem{lemma}[rema]{Lemma}
    
    \theoremstyle{definition}\newtheorem{exam}[rema]{Example}
    \theoremstyle{definition}

\begin{document}
\bibliographystyle{alpha}

\newcommand{\nordcirc}{\mbox{\small $\genfrac{}{}{0pt}{}{\circ}{\circ}$}}
\newcommand{\N}{\mathbb{N}}
\newcommand{\Z}{\mathbb{Z}}
\newcommand{\Q}{\mathbb{Q}}
\newcommand{\R}{\mathbb{R}}
\newcommand{\C}{\mathbb{C}}
\newcommand{\gvmu}{V_{\mathfrak{g}}(\ell,U)}
\newcommand{\imu}{L_{\mathfrak{g}}(\ell,U)}
\newcommand{\gvmzero}{V_{\mathfrak{g}}(\ell,0)}
\newcommand{\imzero}{L_{\mathfrak{g}}(\ell,0)}
\newcommand{\gvmlambda}{V_{\mathfrak{g}}(\ell,\lambda)}
\newcommand{\imlambda}{L_{\mathfrak{g}}(\ell,\lambda)}
\newcommand{\Wo}{W^{(1)}}
\newcommand{\Wt}{W^{(2)}}
\newcommand{\Wth}{W^{(3)}}
\newcommand{\Wi}{W^{(i)}}
\newcommand{\wo}{w_{(1)}}
\newcommand{\wt}{w_{(2)}}
\newcommand{\wth}{w_{(3)}}
\newcommand{\wi}{w_{(i)}}
\newcommand{\g}{\mathfrak{g}}
\newcommand{\ghat}{\widehat{\mathfrak{g}}}
\newcommand{\ilambdaone}{L_{\mathfrak{g}}(\ell,\lambda_1)}
\newcommand{\ilambdatwo}{L_{\mathfrak{g}}(\ell,\lambda_2)}
\newcommand{\ilambdathree}{L_{\mathfrak{g}}(\ell,\lambda_3)}
\newcommand{\T}{T^{(\ell)}_{\lambda_1, \lambda_2}}
\newcommand{\ittensprod}{(W_1\boxtimes_{P(z_1-z_2)} W_2)\boxtimes_{P(z_2)} W_3}
\newcommand{\prodtensprod}{W_1\boxtimes_{P(z_1)} (W_2\boxtimes_{P(z_2)} W_3)}
\newcommand{\gvmone}{V_{\mathfrak{g}}(\ell,U_1)}
\newcommand{\gvmtwo}{V_{\mathfrak{g}}(\ell,U_2)}
\newcommand{\gvmthree}{V_{\mathfrak{g}}(\ell,U_3)}
\newcommand{\gvmp}{V_{\mathfrak{sl}_2}(\ell,p)}
\newcommand{\gvmq}{V_{\mathfrak{sl}_2}(\ell,q)}
\newcommand{\gvmr}{V_{\mathfrak{sl}_2}(\ell,r)}
\newcommand{\res}{\mathrm{Res}}
\newcommand{\nn}{\nonumber\\}
\newcommand{\Id}{\mathrm{Id}}

\begin{abstract}
 We find sufficient conditions for the construction of vertex algebraic intertwining operators, among generalized Verma modules for an affine Lie algebra $\widehat{\mathfrak{g}}$, from $\mathfrak{g}$-module homomorphisms. When $\mathfrak{g}=\mathfrak{sl}_2$, these results extend previous joint work with J. Yang, but the method used here is different. Here, we construct intertwining operators by solving Knizhnik-Zamolodchikov equations for three-point correlation functions associated to $\widehat{\mathfrak{g}}$, and we identify obstructions to the construction arising from the possible non-existence of series solutions having a prescribed form.
\end{abstract}

\maketitle

\numberwithin{equation}{section}

\section{Introduction}

This paper extends the results of \cite{MY} on intertwining operators among generalized Verma modules for $\widehat{\mathfrak{sl}}_2$ to any (untwisted) affine Lie algebra $\ghat$, where $\g$ is a finite-dimensional simple Lie algebra over $\C$. For any level $\ell\neq -h^\vee$, where $h^\vee$ is the dual Coxeter number of $\g$, the generalized Verma $\ghat$-module $\gvmzero$, induced from the one-dimensional $\g$-module and on which the canonical central element of $\ghat$ acts by $\ell$, is a vertex operator algebra \cite{FZ}. Then any generalized Verma module $\gvmu$ induced from a finite-dimensional $\g$-module $U$ is a $\gvmzero$-module; more generally, $\gvmu$ is an $\N$-gradable weak $\gvmzero$-module if $U$ is infinite dimensional.

Intertwining operators among a triple of modules for a vertex operator algebra $V$ are fundamental in the study of tensor categories of $V$-modules (see the review article \cite{HL2}). Indeed, the tensor product of $V$-modules $W_1$ and $W_2$ (if it exists) is the $V$-module $W_1\boxtimes W_2$ such that $\mathrm{Hom}_V(W_1\boxtimes W_2, W_3)$ is naturally isomorphic to the space $\mathcal{V}^{W_3}_{W_1\,W_2}$ of intertwining operators of type $\binom{W_3}{W_1\,W_2}$ for any $V$-module $W_3$. In particular, the fusion rule $\mathcal{N}^{W_3}_{W_1, W_2}=\dim\mathcal{V}^{W_3}_{W_1\,W_2}$ is the number of linearly independent homomorphisms from $W_1\boxtimes W_2$ to $W_3$. While it can be hard to determine when a category of $V$-modules closes under tensor products, a result of Miyamoto \cite{Mi} shows that two $V$-modules satisfying the $C_1$-cofiniteness condition have a $C_1$-cofinite tensor product. For $V=\gvmzero$, generalized Verma modules induced from finite-dimensional $\g$-modules are $C_1$-cofinite, so their tensor products do exist.

Beyond $C_1$-cofinite modules, it is expected that affine Lie (super)algebras at admissible rational levels should admit braided tensor categories of modules that include relaxed highest-weight modules. These $\ghat$-modules are quotients of generalized Verma modules induced from possibly infinite-dimensional irreducible weight $\g$-modules, so they are $\mathbb{N}$-gradable but may have infinite-dimensional conformal weight spaces. In \cite{Ri, CR1, CR2}, Ridout and Creutzig developed a method for predicting tensor products of relaxed highest-weight modules and their spectral flows, using a Verlinde formula based on characters. But rigorous intertwining operator constructions are lacking, except for some examples in the simplest case $\g=\mathfrak{sl}_2$ \cite{Ad}.

Here, we develop a construction method for intertwining operators that we expect will be useful for understanding $\ghat$-module categories at any level $\ell\neq-h^\vee$, including admissible ones. We focus on intertwining operators among $\ghat$-modules of type $\binom{W_3}{\gvmone\,\gvmtwo}$ where $W_3=\bigoplus_{n\in\N} W_3(n)$ is $\N$-gradable and $U_1$, $U_2$, and $W_3(0)$ are, for example, irreducible weight $\g$-modules with finite-dimensional weight spaces. The construction works under conditions that are easy to check in many examples: in Section \ref{sec:sl2} we derive fusion rules for some $\widehat{\mathfrak{sl}}_2$-modules induced from both finite- and infinite-dimensional $\mathfrak{sl}_2$-modules. 

To motivate the main result, we note that a first guess for the tensor product (if it exists) of generalized Verma modules $\gvmone$ and $\gvmtwo$ might be $V_\g(\ell,U_1\otimes U_2)$. By the universal property of induced $\ghat$-modules, this would imply that any $\g$-homomorphism $U_1\otimes U_2\rightarrow U_3$ for $U_3$ a suitable $\g$-module induces an intertwining operator of type $\binom{\gvmthree}{\gvmone\,\gvmtwo}$. However, our results show that the reality is more interesting: we only get intertwining operators from $\g$-module homomorphisms under certain conditions which at least sometimes are necessary. For example, here is a version of the main Theorem \ref{maintheorem2}; in the statement, $h_3$ is the conformal weight of $W_3(0)$ and $C_{U_1\otimes U_2}$ is a Casimir operator on $U_1\otimes U_2$:
\begin{theo}\label{thm:main_theo}
Suppose $W_3=\bigoplus_{n\in\N} W_3(n)$ is an $\N$-gradable weak $\gvmzero$-module and $U_1$, $U_2$, $W_3(0)$ are irreducible weight $\g$-modules with finite-dimensional weight spaces. Then there is a linear isomorphism
 \begin{equation*}
  \mathcal{V}^{W_3}_{\gvmone\,\gvmtwo}\rightarrow\mathrm{Hom}_\g(U_1\otimes U_2,W_3(0))
 \end{equation*}
provided that $(\ell+h^\vee)(h_3+N)-\frac{1}{2} C_{U_1\otimes U_2}$ is invertible on $U_1\otimes U_2$ for all $N\in\Z_+$.
\end{theo}

In Section \ref{sec:obstructions}, we will consider whether non-invertibility of $(\ell+h^\vee)(h_3+N)-\frac{1}{2} C_{U_1\otimes U_2}$ truly obstructs the existence of intertwining operators. In fact there is no obstruction if $W_3$ is the contragredient of a generalized Verma module, but we already showed in \cite{MY} that if $W_3$ is a generalized Verma module for $\widehat{\mathfrak{sl}}_2$, there can be obstructions arising from singular vectors in $W_3$. Here, the proof of Theorem \ref{thm:main_theo} yields a construction of candidates for singular vectors in $W_3$ if $(\ell+h^\vee)(h_3+N)-\frac{1}{2} C_{U_1\otimes U_2}$ is non-invertible for some $N\in\Z_+$ (see Theorem \ref{thm:sing_vect_candidate}).

The proof of Theorem \ref{thm:main_theo} is very different from that of the similar \cite[Theorem 6.1]{MY} (where $\g=\mathfrak{sl}_2$). In \cite{MY}, we adapted the method of \cite{Li2}, for constructing intertwining operators of type $\binom{V_\g(\ell,U_3)'}{\gvmone\,\gvmtwo}$ where the third module is the contragredient of a generalized Verma module, to the case of three generalized Verma modules. But for this to work we had to assume that the third module $\gvmthree$ was not too different from a contragredient (specifically, we assumed $\gvmthree$ had a composition series of length $2$). The method used here is better for modules that are generated by their lowest conformal weight spaces, such as generalized Verma modules. The key observation is that the $L(-1)$-derivative property for intertwining operators implies the restriction of an intertwining operator of type $\binom{W_3}{\gvmone\,\gvmtwo}$ to $U_1\otimes U_2$ satisfies a differential equation, essentially a Knizhnik-Zamolodchikov (KZ) equation for three-point correlation functions. Thus, we first try to solve the KZ equation with a series solution ansatz to obtain a linear map
\begin{equation*}
 \mathcal{Y}: U_1\otimes U_2\rightarrow W_3\lbrace x\rbrace.
\end{equation*}
Potential obstructions arise when the series coefficients cannot be computed recursively from the initial data of a $\g$-module homomorphism $U_1\otimes U_2\rightarrow W_3(0)$, but if $\mathcal{Y}$ can be constructed, it uniquely extends to an intertwining operator exactly as in \cite{MY}.

We anticipate that it will be possible to extend Theorem \ref{thm:main_theo} in several directions. First, it is often important to determine all intertwining operators among irreducible $\ghat$-modules of type $\binom{L_{\g}(\ell,U_3)}{L_{\g}(\ell,U_1)\,L_{\g}(\ell, U_2)}$, where each $L_{\g}(\ell,U_i)$ is the unique simple quotient of $V_{\g}(\ell,U_i)$ for an irreducible $\g$-module $U_i$. This amounts to determining when intertwining operators of type $\binom{L_{\g}(\ell,U_3)}{V_{\g}(\ell,U_1)\,V_{\g}(\ell, U_2)}$, constructed as in Theorem \ref{thm:main_theo}, remain well defined on the irreducible quotients; for an example of a result of this type in the case $\g=\mathfrak{sl}_2$, see \cite[Theorem 7.1]{MY}. A second problem is constructing intertwining operators among non-$\mathbb{N}$-gradable $\gvmzero$-modules, especially spectral flows of relaxed highest-weight modules at admissible levels. Although the proof of Theorem \ref{thm:main_theo} is not well-suited to handling modules that are not $\mathbb{N}$-gradable, it is sometimes possible to compute fusion rules among spectral flow modules by relating them to intertwining operators among $\mathbb{N}$-gradable modules: for example, this was done for the $\beta\gamma$-ghost vertex algebra in \cite{AP}.

Finally, Theorem \ref{thm:main_theo} extends almost immediately to intertwining operators among untwisted modules for affine Lie superalgebras, using the corresponding KZ equations for superalgebras. However, in the superalgebra setting it is natural to also consider twisted intertwining operators among parity-twisted modules (the Ramond sector in physics terminology). For example, when $\gvmzero$ is the symplectic fermion vertex operator superalgebra associated to an abelian Lie superalgebra $\g$, a classification of twisted intertwining operators will be useful for rigorously proving the formulas in \cite[Table 1]{Ru} for four-point correlation functions involving Ramond $\gvmzero$-modules. As proving construction theorems for twisted intertwining operators will first require generalizing Theorem \ref{intwopext} below  (which is based on commutativity properties specific to untwisted intertwining operators), we will study affine Lie superalgebras in future work.

We now summarize the remaining contents of this paper. In Section \ref{sec:affine}, we recall definitions and notation for affine Lie algebras. In Section \ref{sec:intw_op}, we recall the definition of intertwining operator and prove our main construction theorems for intertwining operators among $\gvmzero$-modules. In Section \ref{sec:obstructions}, we treat the question of when the conditions of Theorem \ref{thm:main_theo} are necessary. Finally in Section \ref{sec:sl2}, we present new examples of intertwining operators when $\g=\mathfrak{sl}_2$ and compare with previous results from \cite{MY}.

\section{Affine Lie algebras}\label{sec:affine}

Let $\g$ be a finite-dimensional simple Lie algebra over $\C$ with non-zero invariant bilinear form $\langle\cdot,\cdot\rangle$ scaled so that long roots have square length $2$. Then the affine Lie algebra is
\begin{equation*}
 \ghat=\g\otimes\C[t,t^{-1}]\oplus\C\mathbf{k}
\end{equation*}
with $\mathbf{k}$ central and all other brackets determined by
\begin{equation}\label{affalgcomm}
 [g\otimes t^m, h\otimes t^n]=[g,h]\otimes t^{m+n}+m\langle g,h\rangle\delta_{m+n,0}\mathbf{k}
\end{equation}
for $g,h\in\g$ and $m,n\in\Z$. The Lie algebra $\ghat$ has the decomposition
\begin{equation*}
 \ghat=\ghat_-\oplus\ghat_0\oplus\ghat_+
\end{equation*}
where
\begin{equation*}
 \ghat_\pm=\bigoplus_{n\in\pm\Z_+}\g\otimes t^n,\hspace{4em}\ghat_0=\g\otimes t^0\oplus\C\mathbf{k}.
\end{equation*}
For $g\in\g$ and $m\in\Z$, we will use $g(m)$ to denote the action of $g\otimes t^m$ on a $\ghat$-module.

If $U$ is a $\g$-module, then $U$ becomes a $\ghat_0\oplus\ghat_+$-module on which $\ghat_+$ acts trivially and $\mathbf{k}$ acts as some scalar $\ell\in\C$.
The generalized Verma $\ghat$-module is then the induced module
\begin{equation*}
 \gvmu=U(\ghat)\otimes_{U(\ghat_0\oplus\ghat_+)} U.
\end{equation*}
We say that $\ell$ is the \textit{level} of $\gvmu$. Since the Poincar\'{e}-Birkhoff-Witt Theorem implies
\begin{equation*}
 \gvmu\cong U(\ghat_-)\otimes_{\C} U
\end{equation*}
as vector spaces, $\gvmu$ is spanned by vectors of the form
\begin{equation*}
 g_1(-n_1)\cdots g_k(-n_k)u
\end{equation*}
for $g_i\in\g$, $n_i\in\Z_+$, and $u\in U$.

\begin{rema}
 If $\lambda$ is a weight of $\g$ and $L_\lambda$ is the associated irreducible highest-weight $\g$-module, we use $\gvmlambda$ to denote the generalized Verma module induced from $L_\lambda$. In particular, $\gvmzero$ denotes the generalized Verma module induced from the one-dimensional $\g$-module $\C\mathbf{1}$.
\end{rema}

For any level $\ell$, $\gvmzero$ is a vertex algebra with vacuum $\mathbf{1}$ \cite{FZ} (see also \cite[Section 6.2]{LL}). The vertex algebra $\gvmzero$ is generated by the vectors $g(-1)\mathbf{1}$ for $g\in\g$, with vertex operators
\begin{equation*}
 Y(g(-1)\mathbf{1},x)=g(x)=\sum_{n\in\Z} g(n)\,x^{-n-1}.
\end{equation*}
Moreover, the same vertex operators acting on any generalized Verma module $\gvmu$ give it the structure of an $\N$-gradable weak $\gvmzero$-module: $\gvmu=\bigoplus_{n\in\N} \gvmu(n)$ for
\begin{equation*}
 \gvmu(n)=\mathrm{span}\lbrace g_1(-n_1)\cdots g_k(-n_k)u\,\vert\,u\in U, g_i\in\g, n_i\in\Z_+, n_1+\ldots+n_k=n\rbrace.
\end{equation*}

Let $\lbrace \gamma_i\rbrace_{i=1}^{\mathrm{dim}\,\g}$ be an orthonormal basis for $\g$ with respect to the nondegenerate form $\langle\cdot,\cdot\rangle$. The Casimir element $\sum_{i=1}^{\mathrm{dim}\,\g} \gamma_i^2$ associated to $\langle\cdot,\cdot\rangle$ acts in the adjoint representation $\g$ by $2 h^\vee$, where $h^\vee$ is the dual Coxeter number of $\g$. Then if $\ell\neq -h^\vee$, $\gvmzero$ is a vertex operator algebra with conformal vector
\begin{equation*}
 \omega=\dfrac{1}{2(\ell+h^\vee)}\sum_{i=1}^{\mathrm{dim}\,\g} \gamma_i(-1)^2\mathbf{1}.
\end{equation*}
Writing $Y(\omega,x)=\sum_{n\in\Z} L(n)\,x^{-n-2}$ as usual, we have (see \cite[Theorem 6.2.16]{LL}) that
\begin{equation}\label{Lgcomm}
 [L(m),g(n)]=-n g(m+n)
\end{equation}
for any $g\in\g$ and $m,n\in\Z$. From the definition of $\omega$, it also follows that
\begin{equation}\label{L0}
 L(0)=\dfrac{1}{2(\ell+h^\vee)}\sum_{i=1}^{\mathrm{dim}\,\g} \gamma_i(0)^2+\dfrac{1}{\ell+h^\vee}\sum_{i=1}^{\mathrm{dim}\,\g}\sum_{n>0} \gamma_i(-n) \gamma_i(n)
\end{equation}
and
\begin{equation}\label{L-1}
 L(-1)=\dfrac{1}{\ell+h^\vee}\sum_{i=1}^{\mathrm{dim}\,\g}\sum_{n\geq 0} \gamma_i(-n-1)\gamma_i(n).
\end{equation}
\begin{rema}\label{geneigenU}
By \eqref{Lgcomm} and \eqref{L0}, any vector of the form
\begin{equation*}
 g_1(-n_1)\cdots g_k(-n_k)\mathbf{1}\in\gvmzero
\end{equation*}
for $g_i\in\g$ and $n_i\in\Z_+$ has conformal weight $n_1+\ldots+n_k$. More generally, for any weight $\lambda$, \eqref{L0} shows that $L(0)$ acts on $L_\lambda=\gvmlambda(0)$ by the scalar
 \begin{equation}\label{hlambdal}
  h_{\lambda,\ell}=\dfrac{1}{2(\ell+h^\vee)}\langle\lambda,\lambda+2\rho\rangle
 \end{equation}
where $\rho$ is the sum of the fundamental weights of $\g$ (see \cite[Section 22]{Hu}). Then by \eqref{Lgcomm}, $\gvmlambda(n)$ is the conformal weight space of $\gvmlambda$ with $L(0)$-eigenvalue $h_{\lambda,\ell}+n$.
\end{rema}
\begin{rema}
 The $m=n=0$ case of \eqref{Lgcomm} implies that the $L(0)$-generalized eigenspaces of a weak $\gvmzero$-module are $\g$-modules.
\end{rema}

\section{Construction of intertwining operators}\label{sec:intw_op}

For a general vector space $W$, we use $W\{x\}$ to denote the vector space of formal series of the form $\sum_{n \in \C} w_n\,x^n$, $w_n \in W$. We recall from \cite{FHL} (see also \cite{HL1, HLZ2}) the definition of intertwining operator among a triple of modules for a vertex operator algebra:

\begin{defi}{\rm
Let $W_1$, $W_2$ and $W_3$
be weak modules for a vertex operator algebra $V$. An {\it intertwining
operator} of type $\binom{W_3}{W_1\,W_2}$ is a linear map
\begin{align*}
\mathcal{Y}: W_1\otimes W_2 & \rightarrow W_3\{x\}\\
w_{1}\otimes w_{2}&\mapsto \mathcal{Y}(w_{1},x)w_{2}=\sum_{h\in
{\mathbb C}}(w_{1})_h
w_{2}\,x^{-h-1}\in W_3\{x\}
\end{align*}
satisfying the
following conditions:

\begin{enumerate}

\item  {\it Lower truncation}: for $w_{1}\in W_1$, $w_{2}\in W_2$, and $h\in
\mathbb{C}$, $(w_{1})_{h+n}w_{2}=0$ for $n\in {\mathbb
N}$ sufficiently large.

\item The {\it Jacobi identity}: for $v\in V$ and $w_{1}\in W_1$,
\begin{align}\label{jacobi}
 x^{-1}_0\delta \left( \frac{x_1-x_2}{x_0}\right)
Y_{W_{3}}(v,x_1)\mathcal{Y}(w_{1}, & \, x_2)  -  x^{-1}_0\delta \left( \frac{x_2-x_1}{-x_0}\right)
\mathcal{Y}(w_{1},x_2)Y_{W_{2}}(v,x_1)\nonumber\\
& =  x^{-1}_1\delta \left( \frac{x_2+x_0}{x_1}\right) \mathcal{
Y}(Y_{W_{1}}(v,x_0)w_{1},x_2).
\end{align}

\item The {\em $L(-1)$-derivative property:} for
$w_{1}\in W_1$,
\begin{equation}\label{L(-1)dev}
\mathcal{Y}(L(-1)w_{1},x)=\frac d{dx}\mathcal{Y}(w_{1},x).
\end{equation}
\end{enumerate}}
\end{defi}
\begin{rema}
We denote the vector space of intertwining operators of type $\binom{W_3}{W_1\,W_2}$ by $\mathcal{V}^{W_3}_{W_1\,W_2}$ and the corresponding \textit{fusion rule} is $\mathcal{N}^{W_3}_{W_1\,W_2}=\mathrm{dim}\,\mathcal{V}^{W_3}_{W_1\,W_2}$.
\end{rema}

Taking $W_1$, $W_2$, and $W_3$ to be $\N$-gradable weak $\gvmzero$-modules for some level $\ell$, suppose $\mathcal{Y}$ is an intertwining operator of type $\binom{W_3}{W_1\,W_2}$. We will need some consequences of the Jacobi identity and $L(-1)$-derivative property in this setting. First, the coefficient of $x_0^{-1} x_1^{-n-1}$ for $n\in\Z$ in \eqref{jacobi}, when $v=g(-1)\mathbf{1}$, $g\in\g$, is the commutator formula
\begin{align*}\label{commform}
 [g(n),\mathcal{Y}(w_1,x_2)] & =\mathrm{Res}_{x_0}\, (x_2+x_0)^n\mathcal{Y}(g(x_0)w_1,x_2)=\sum_{i\geq 0} \binom{n}{i} x_2^{n-i}\mathcal{Y}(g(i)w_1,x_2).
%  & =\sum_{i\geq 0}\dfrac{(-1)^i}{i!}\left(\dfrac{\partial}{\partial x_1}\right)^i\left(x_2^{-1}\delta\left(\dfrac{x_1}{x_2}\right)\right)\mathcal{Y}(g(i)w_1,x_2)
\end{align*}
If $w_1\in W_1$ satisfies $g(i)w_1=0$ for all $i>0$, we get
% \begin{equation*}\label{commformcoeff}
%  [g(x_1),\mathcal{Y}(w_1,x_2)]=x_2^{-1}\delta\left(\dfrac{x_1}{x_2}\right)\mathcal{Y}(g(0)w_1,x_2).
% \end{equation*}
% In particular, taking the coefficient of $x_1^{-n-1}$ and changing $x_2$ to $x$, we have
\begin{equation}\label{gcomp}
 [g(n),\mathcal{Y}(w_1,x)]=x^n\mathcal{Y}(g(0)w_1,x).
\end{equation}
Next, the iterate formula is the coefficient of $x_0^{-n-1} x_1^{-1}$ in \eqref{jacobi} for $v=g(-1)\mathbf{1}$ and $n\in \Z$:  
% \begin{equation*}
%  \mathcal{Y}(g(x_0)w_1,x_2)=\mathrm{Res}_{x_1}\left(x^{-1}_0\delta \left( \frac{x_1-x_2}{x_0}\right)
% g(x_1)\mathcal{Y}(w_1,x_2)  -  x^{-1}_0\delta \left( \frac{x_2-x_1}{-x_0}\right)
% \mathcal{Y}(w_1,x_2)g(x_1)\right)
% \end{equation*}
% for $w\in W_1$. Further taking the coefficient of $x_0^{-n-1}$, 
\begin{equation*}\label{itform}
 \mathcal{Y}(g(n)w_1,x_2)=\mathrm{Res}_{x_1}\left((x_1-x_2)^{n} g(x_1)\mathcal{Y}(w_1,x_2)-(-x_2+x_1)^{n}\mathcal{Y}(w_1,x_2)g(x_1)\right).
\end{equation*}
 The case $n=-1$ yields
\begin{align}\label{nord}
 \mathcal{Y}(g(-1)w,x_2) &  =\sum_{i\geq 0} g(-i-1) x_2^i\mathcal{Y}(w,x_2)+\sum_{i\geq 0} x_2^{-i-1}\mathcal{Y}(w,x_2)g(i)\nonumber\\
 & =g(x_2)^+\mathcal{Y}(w,x_2)+\mathcal{Y}(w,x_2)g(x_2)^-,
\end{align}
where $g(x)^{\pm}$ denote the non-singular and singular parts of $g(x)$, respectively.

% \begin{rema}\label{commitequivtojac}
%  The commutator and iterate formulas \eqref{commform} and \eqref{itform} are together equivalent to the Jacobi identity for $v=g(-1)\mathbf{1}$ and $w$ (see \cite{LL} for the special case in which $\mathcal{Y}$ is a module vertex operator).
% \end{rema}

Now for the $L(-1)$-derivative property: when $w_{1}\in W_1$ satisfies $g(i)w_{1}=0$ for $g\in\g$ and $i>0$, \eqref{L-1}, \eqref{L(-1)dev}, and \eqref{nord} imply
\begin{align*}
 (\ell+h^\vee)  \dfrac{d}{dx}\mathcal{Y}(w_{1},x)w_{2} &=\sum_{i=1}^{\mathrm{dim}\,\g}\mathcal{Y}(\gamma_i(-1)\gamma_i(0)w_{1},x)w_{2}\nonumber\\
 & =\sum_{i=1}^{\mathrm{dim}\,\g} \left(\gamma_i(x)^+\mathcal{Y}(\gamma_i(0)w_{1},x)w_{2}+\mathcal{Y}(\gamma_i(0)w_{1},x)\gamma_i(x)^-w_{2}\right)
\end{align*}
for any $w_{2}\in W_2$. If also $g(i)w_{2}=0$ for $g\in\g$ and $i>0$, we then have:
\begin{propo}\label{KZpropo}
 If $w_{1}\in W_1$, $w_{2}\in W_2$ satisfy $g(i)w_{1}, g(i)w_{2}=0$ for $g\in\g$ and $i>0$, then
 \begin{equation*}
  (\ell+h^\vee)\dfrac{d}{dx}\mathcal{Y}(w_{1},x)w_{2}=\sum_{i=1}^{\mathrm{dim}\,\g} \left(x^{-1}\mathcal{Y}(\gamma_i(0)w_{1},x)\gamma_i(0)w_{2}+\gamma_i(x)^+\mathcal{Y}(\gamma_i(0)w_{1},x)w_{2}\right).
 \end{equation*}
\end{propo}
\begin{rema}
We shall construct intertwining operators of type $\binom{W_3}{W_1\,W_2}$ from solutions to the differential equation of Proposition \ref{KZpropo}, which is basically a Knizhnik-Zamolodchikov equation \cite{KZ} for three-point correlation functions in conformal field theory based on $\ghat$. 
\end{rema}

For $i=1,2,3$, let $U_i$ denote the degree-zero subspace of the $\N$-gradable weak $V$-module $W_i$; each $U_i$ is a $\g$-module. Then if $\mathcal{Y}$ is an intertwining operator of type $\binom{W_3}{W_1\,W_2}$, Proposition \ref{KZpropo} shows that
\begin{equation*}
\mathcal{Y}\vert_{U_1\otimes U_2}=\mathcal{Y}(\cdot\otimes\cdot,x): U_1\otimes U_2\rightarrow W_3\lbrace x\rbrace
\end{equation*}
satisfies
\begin{equation*}
 (\ell+h^\vee)\dfrac{d}{dx} \mathcal{Y}(u_{1}\otimes u_{2},x)=x^{-1}\mathcal{Y}(C_{U_1,U_2}(u_{1}\otimes u_{2}),x)+\sum_{i=1}^{\mathrm{dim}\,\g} \gamma_i(x)^+ \mathcal{Y}(\gamma_i\cdot u_{1}\otimes u_{2},x),
\end{equation*}
where for $u_1\in U_1$, $u_2\in U_2$,
\begin{equation}\label{cu1u2orig}
 C_{U_1,U_2}(u_{1}\otimes u_{2})=\sum_{i=1}^{\mathrm{dim}\,\g} \gamma_i\cdot u_{1}\otimes \gamma_i\cdot u_{2}
\end{equation}
Using $C_U$ to denote the action of the Casimir element of $(\g,\langle\cdot,\cdot\rangle)$ on a $\g$-module $U$, we have
\begin{equation}\label{cu1u2}
 C_{U_1,U_2}=\dfrac{1}{2}\left(C_{U_1\otimes U_2}-C_{U_1}\otimes 1_{U_2}-1_{U_1}\otimes C_{U_2}\right),
\end{equation}
so $C_{U_1,U_2}$ is a $\g$-endomorphism of $U_1\otimes U_2$.
% \begin{rema}\label{rema:cu1u2_diag}
% When $U_1$ and $U_2$ are finite dimensional, \eqref{cu1u2} shows $C_{U_1,U_2}$ is diagonalizable: since finite-dimensional $\g$-modules are completely reducible, we may assume $U_1$ and $U_2$ are irreducible. Then $C_{U_1}\otimes 1_{U_2}$ and $1_{U_1}\otimes C_{U_2}$ are scalars on  all $U_1\otimes U_2$ while $C_{U_1\otimes U_2}$ is diagonalizable since it acts as a scalar on each irreducible summand of $U_1\otimes U_2$.
% \end{rema}

We now present our construction theorems for intertwining operators among $\gvmzero$-modules. The first is from \cite{MY}; it does not require the $\g$-modules $U_i$ to be finite dimensional:
\begin{theo}[\cite{MY}, Theorem 6.2]\label{intwopext}
 Suppose $\mathcal{Y}(\cdot,x)\cdot: U_1\otimes U_2\rightarrow W_3\lbrace x\rbrace$ is a lower-truncated linear map satisfying
\begin{equation}\label{comm1}
 [g(n),\mathcal{Y}(u,x)]=x^n\mathcal{Y}(g(0)u,x)
\end{equation}
for $g\in\mathfrak{g}$, $n\geq 0$, and
\begin{equation}\label{l0comm}
 [L(0),\mathcal{Y}(u,x)]=x\dfrac{d}{dx}\mathcal{Y}(u,x)+\mathcal{Y}(L(0)u,x).
\end{equation}
Then $\mathcal{Y}$ has a unique extension to an intertwining operator of type $\binom{W_3}{\gvmone\,\gvmtwo}$.
\end{theo}

Conversely, by the commutator and $L(-1)$-derivative formulas, any intertwining operator of type $\binom{W_3}{\gvmone\,\gvmtwo}$ restricted to $U_1\otimes U_2$ satisfies \eqref{comm1} and \eqref{l0comm}. To construct lower-truncated linear maps as in Theorem \ref{intwopext}, an ansatz for the shape of the formal series $\mathcal{Y}$ will help. Thus, we now assume that $L(0)$ acts on each $\g$-module $U_i$ as a scalar $h_i\in\C$; equivalently by \eqref{L0}, $C_{U_i}$ acts as $2(\ell+h^\vee)h_i$. For example, $U_i$ could be a not-necessarily-finite-dimensional irreducible weight $\g$-module with finite-dimensional weight spaces.
Set $h=h_3-h_1-h_2$; under our assumption, \eqref{l0comm} is equivalent to
\begin{equation*}
 \mathcal{Y}(u_1,x)u_2=\sum_{m\in\Z}\mathcal{Y}_m(u_1\otimes u_2)\,x^{h+m}
\end{equation*}
for $u_1\in U_1$, $u_2\in U_2$, with $\mathcal{Y}_m: U_1\otimes U_2\rightarrow W_3(m)$ for $m\in\Z$. Then \eqref{comm1} is equivalent to
\begin{equation}\label{commcomp}
 g(n)\mathcal{Y}_m(u_1\otimes u_2)=\mathcal{Y}_{m-n}(g(0)u_1\otimes u_2)+\mathcal{Y}_m(u_1\otimes g(n)u_2)
\end{equation}
for $g\in\g$, $u_1\in U_1$, $u_2\in U_2$, $m\in\Z$, and $n\geq 0$. 

Each $\mathcal{Y}_m$ is a $\g$-module homomorphism by the $n=0$ case of \eqref{commcomp}, so $\mathcal{Y}\mapsto\mathcal{Y}_0$ defines a linear map $\mathcal{V}^{W_3}_{\gvmone\,\gvmtwo}\rightarrow\mathrm{Hom}_\g(U_1\otimes U_2,U_3)$. Conversely, we will construct intertwining operators starting from such $\g$-module homomorphisms using the following theorem:
\begin{theo}\label{maintheorem1}
 Suppose that for all $N\in\Z_+$, the $\g$-module endomorphism
$$(\ell+h^\vee)(h+N)-C_{U_1,U_2}$$
of $U_1\otimes U_2$ is invertible. Then for any $f\in\mathrm{Hom}_\g(U_1\otimes U_2,U_3)$, there are unique linear maps
 \begin{equation*}
  \mathcal{Y}_m: U_1\otimes U_2\rightarrow W_3(m)
 \end{equation*}
for $m\in\Z$ such that $\mathcal{Y}_0=f$ and \eqref{commcomp} holds.
\end{theo}
\begin{proof}
 Since $W_3$ is $\N$-graded, we must have $\mathcal{Y}_m=0$ for $m<0$. Now, if the desired linear maps exist for $m\geq 0$, then Theorem \ref{intwopext} implies that
 \begin{equation*}
  \mathcal{Y}(\cdot,x)\cdot=\sum_{m\geq 0} \mathcal{Y}_m\, x^{h+m}: U_1\otimes U_2\rightarrow W_3\lbrace x\rbrace
 \end{equation*}
extends to an intertwining operator of type $\binom{W_3}{\gvmone\,\gvmtwo}$. So by Proposition \ref{KZpropo},
\begin{equation*}
 (\ell+h^\vee)\dfrac{d}{dx}\mathcal{Y}(u_1,x)u_2=\sum_{i=1}^{\mathrm{dim}\,\g}\left(x^{-1}\mathcal{Y}(\gamma_i\cdot u_1,x)(\gamma_i\cdot u_2)+\gamma_i(x)^+\mathcal{Y}(\gamma_i\cdot u_1,x)u_2\right)
\end{equation*}
for $u_1\in U_1$ and $u_2\in U_2$. In component form, this is
\begin{equation}\label{KZcomp}
 \mathcal{Y}_m\big([(\ell+h^\vee)(h+m)-C_{U_1,U_2}](u_1\otimes u_2)\big)=\sum_{i=1}^{\mathrm{dim}\,\g}\sum_{k=1}^{m}\gamma_i(-k)\mathcal{Y}_{m-k}(\gamma_i\cdot u_1\otimes u_2)
\end{equation}
for $m\geq 0$. To prove the theorem, it is enough to show that \eqref{KZcomp} has a unique solution for $m> 0$ given $\mathcal{Y}_0=f$, and that this solution satisfies \eqref{commcomp}.

We first show that $\mathcal{Y}_0=f$ satisfies the $m=0$ case of \eqref{KZcomp}. Since $C_{U_i} = 2(\ell+h^\vee)h_i$ for $i=1,2,3$ (recall \eqref{L0}), \eqref{cu1u2} implies
\begin{align}\label{eqn:Casimir_relation}
 (\ell+h^\vee)h- C_{U_1,U_2} & =(\ell+h^\vee)(h_3-h_1-h_2)-\frac{1}{2}(C_{U_1\otimes U_2}-C_{U_1}\otimes 1_{U_2}-1_{U_1}\otimes C_{U_2})\nonumber\\
 & = (\ell+h^\vee)h_3-\frac{1}{2}C_{U_1\otimes U_2}.
\end{align}
Then because Casimir operators commute with $\g$-homomorphisms, we have
\begin{equation}\label{holdsfor0}
 f\big([(\ell+h^\vee)h-C_{U_1,U_2}](u_1\otimes u_2)\big)=(\ell+h^\vee)h_3 f(u_1\otimes u_2)-\frac{1}{2} C_{U_3} f(u_1\otimes u_2)=0
\end{equation}
for $u_1\in U_1$ and $u_2\in U_2$, as required. Now we can use \eqref{KZcomp} to construct $\mathcal{Y}_m$ recursively, since by assumption $(\ell+h^\vee)(h+m)-C_{U_1,U_2}$ is invertible for all $m>0$:
\begin{equation}\label{ym}
 \mathcal{Y}_m(u_1\otimes u_2)=\sum_{i=1}^{\mathrm{dim}\,\g}\sum_{k=1}^m \gamma_i(-k)\mathcal{Y}_{m-k}\left((\gamma_i\otimes 1)\big[(\ell+h^\vee)(h+m)-C_{U_1,U_2}\big]^{-1}(u_1\otimes u_2)\right)
\end{equation}
for $u_1\in U_1$, $u_2\in U_2$. This shows \eqref{KZcomp} has a unique solution for each $m> 0$ given $\mathcal{Y}_0=f$. 

\allowdisplaybreaks

We need to show that $\mathcal{Y}_m$ as given by \eqref{ym} satisfies \eqref{commcomp} for $m\geq 0$. As both sides of \eqref{commcomp} are zero for $n>m$, we may assume $0\leq n\leq m$ and prove \eqref{commcomp} by induction on $m$. The base case $m=0$ is clear because $\mathcal{Y}_0=f$ is a $\g$-module homomorphism, so we assume \eqref{commcomp} holds for all $m$ less than some fixed $M>0$ and prove \eqref{commcomp} for $M$. Since the $\g$-homomorphism $C^{(M)}_{U_1,U_2}=(\ell+h^\vee)(h+M)-C_{U_1,U_2}$ is invertible, it is enough to prove that
\begin{align}
 g(0)\mathcal{Y}_M\big(C^{(M)}_{U_1,U_2}(u_1\otimes u_2)\big)=\mathcal{Y}_{M}\big(C^{(M)}_{U_1,U_2}(g\otimes 1+1\otimes g)(u_1\otimes u_2)\big), \label{g(0)comp} \\
 g(n)\mathcal{Y}_M\big(C^{(M)}_{U_1,U_2}(u_1\otimes u_2)\big)=\mathcal{Y}_{M-n}\big((g\otimes 1)C^{(M)}_{U_1,U_2}(u_1\otimes u_2)\big) \label{g(n)comp} 
\end{align}
for $g\in\g$ and $1\leq n\leq M$. For \eqref{g(0)comp}, we use \eqref{KZcomp}, the induction hypothesis, the commutation relations \eqref{affalgcomm}, and Lemma \ref{lemma1} below to obtain
\begin{align*}
 g & (0)  \mathcal{Y}_M\big(C^{(M)}_{U_1,U_2}(u_1\otimes u_2)\big)  =\sum_{i=1}^{\mathrm{dim}\,\g}\sum_{k=1}^M g(0)\gamma_i(-k)\mathcal{Y}_{M-k}(\gamma_i\cdot u_1\otimes u_2)\nonumber\\
 & =\sum_{i=1}^{\mathrm{dim}\,\g}\sum_{k=1}^M \big(\gamma_i(-k)\mathcal{Y}_{M-k}((g\otimes 1+1\otimes g)(\gamma_i\otimes 1)(u_1\otimes u_2))+ [g,\gamma_i](-k)\mathcal{Y}_{M-k}(\gamma_i\cdot u_1\otimes u_2)\big)\nonumber\\
 & =\sum_{i=1}^{\mathrm{dim}\,\g}\sum_{k=1}^M \gamma_i(-k)\mathcal{Y}_{M-k}\big((\gamma_i\otimes 1)(g\otimes 1+1\otimes g)(u_1\otimes u_2)\big)\nonumber\\
 &\;\;\;\;\;+\sum_{i=1}^{\mathrm{dim}\,\g}\sum_{k=1}^M \big([g,\gamma_i](-k)\mathcal{Y}_{M-k}((\gamma_i\otimes 1)(u_1\otimes u_2))+\gamma_i(-k)\mathcal{Y}_{M-k}(([g,\gamma_i]\otimes 1)(u_1\otimes u_2))\big)\nonumber\\
 & =\mathcal{Y}_M \big(C^{(M)}_{U_1,U_2}(g\otimes 1+1\otimes g)(u_1\otimes u_2)\big)
\end{align*}
for any $u_1\in U_1$ and $u_2\in U_2$.

\begin{lemma}\label{lemma1}
 In $\g\otimes\g$, $\sum_{i=1}^{\mathrm{dim}\,\g} [g,\gamma_i]\otimes\gamma_i=-\sum_{i=1}^{\mathrm{dim}\,\g}\gamma_i\otimes [g,\gamma_i]$ for any $g\in\g$.
\end{lemma}
\begin{proof}
For $g\in\g$, we have $[g,\gamma_i]=\sum_{j=1}^{\mathrm{dim}\,\g} c^j_i \gamma_j$ for each $1\leq i\leq\mathrm{dim}\,\g$, where $c^j_i\in\C$. Then
\begin{equation*}
 \sum_{i=1}^{\mathrm{dim}\,\g} [g,\gamma_i]\otimes\gamma_i=\sum_{i,j=1}^{\mathrm{dim}\,\g} c^j_i (\gamma_j\otimes\gamma_i),
\end{equation*}
while
\begin{equation*}
 \sum_{i=1}^{\mathrm{dim}\,\g}\gamma_i\otimes[g,\gamma_i]=\sum_{i,j=1}^{\mathrm{dim}\,\g} c^j_i(\gamma_i\otimes\gamma_j)=\sum_{i,j=1}^{\mathrm{dim}\,\g}c^i_j (\gamma_j\otimes\gamma_i).
\end{equation*}
But by the invariance of the form $\langle\cdot,\cdot\rangle$ on $\g$, we have
\begin{equation*}
 c^j_i=\langle\gamma_j,[g,\gamma_i]\rangle=\langle [\gamma_j,g],\gamma_i\rangle=-c^i_j,
\end{equation*}
proving the lemma.
\end{proof}

For \eqref{g(n)comp}, we use  \eqref{KZcomp}, the affine Lie algebra commutation relations \eqref{affalgcomm}, the induction hypothesis (formula \eqref{commcomp} for $0\leq m<M$), Lemma \ref{lemma1}, and the induction hypothesis again:
\begin{align}\label{calc1}
 g(n) & \mathcal{Y}_M\big(C^{(M)}_{U_1,U_2}(u_1\otimes u_2)\big)=\sum_{i=1}^{\mathrm{dim}\,\g}\sum_{k=1}^M g(n)\gamma_i(-k)\mathcal{Y}_{M-k}(\gamma_i\cdot u_1\otimes u_2)\nonumber\\
 & =\sum_{i=1}^{\dim\g}\sum_{k=1}^M \gamma_i(-k)g(n)\mathcal{Y}_{M-k}(\gamma_i\cdot u_1\otimes u_2)\nonumber\\
 &\;\;\;\;\;+\sum_{i=1}^{\mathrm{dim}\,\g}\sum_{k=1}^M \big([g,\gamma_i](n-k)+n\langle g,\gamma_i\rangle\delta_{n,k}\ell\big)\mathcal{Y}_{M-k}(\gamma_i\cdot u_1\otimes u_2)\nonumber\\
 &=\sum_{i=1}^{\mathrm{dim}\,\g}\sum_{k=1}^{M-n} \gamma_i(-k)\mathcal{Y}_{M-n-k}(g\cdot(\gamma_i\cdot u_1)\otimes u_2)\nonumber\\
 &\;\;\;\;\;+ \sum_{i=1}^{\dim\g}\left(\sum_{k=n+1}^M [g,\gamma_i](n-k)\mathcal{Y}_{M-k}(\gamma_i\cdot u_1\otimes u_2)+\sum_{k=1}^n [g,\gamma_i](n-k)\mathcal{Y}_{M-k}(\gamma_i\cdot u_1\otimes u_2)\right)\nonumber\\
 &\;\;\;\;\;+n\ell\sum_{i=1}^{\dim\g}\langle g,\gamma_i\rangle\mathcal{Y}_{M-n}(\gamma_i\cdot u_1\otimes u_2)\nonumber\\ 
 & =\sum_{i=1}^{\mathrm{dim}\,\g}\sum_{k=1}^{M-n}\gamma_i(-k)\mathcal{Y}_{M-n-k}(\gamma_i\cdot(g\cdot u_1)\otimes u_2)\nonumber\\
 &\;\;\;\;\;+\sum_{i=1}^{\mathrm{dim}\,\g}\sum_{k=1}^{M-n}\big(\gamma_i(-k)\mathcal{Y}_{M-n-k}([g,\gamma_i]\cdot u_1\otimes u_2)+[g,\gamma_i](-k)\mathcal{Y}_{M-n-k}(\gamma_i\cdot u_1\otimes u_2)\big)\nonumber\\
 &\;\;\;\;\;+\sum_{i=1}^{\mathrm{dim}\,\g}\sum_{k=1}^n [g,\gamma_i](n-k)\mathcal{Y}_{M-k}(\gamma_i\cdot u_1\otimes u_2) +n\ell\mathcal{Y}_{M-n}(g\cdot u_1\otimes u_2)\nonumber\\
& =\sum_{i=1}^{\mathrm{dim}\,\g}\sum_{k=1}^{M-n}\gamma_i(-k)\mathcal{Y}_{M-n-k}(\gamma_i\cdot(g\cdot u_1)\otimes u_2)+\sum_{i=1}^{\mathrm{dim}\,\g}\sum_{k=1}^{n-1} \mathcal{Y}_{M-n}([g,\gamma_i]\cdot(\gamma_i\cdot u_1)\otimes u_2)\nonumber\\
&\;\;\;\;\; +\sum_{i=1}^{\mathrm{dim}\,\g}\mathcal{Y}_{M-n}\big([g,\gamma_i]\cdot(\gamma_i\cdot u_1)\otimes u_2+\gamma_i\cdot u_1\otimes [g,\gamma_i]\cdot u_2\big)+n\ell\mathcal{Y}_{M-n}(g\cdot u_1\otimes u_2)\nonumber\\
 & =\sum_{i=1}^{\mathrm{dim}\,\g}\sum_{k=1}^{M-n}\gamma_i(-k)\mathcal{Y}_{M-n-k}(\gamma_i\cdot(g\cdot u_1)\otimes u_2)+n\sum_{i=1}^{\mathrm{dim}\,\g} \mathcal{Y}_{M-n}([g,\gamma_i]\cdot(\gamma_i\cdot u_1)\otimes u_2)\nonumber\\
&\;\;\;\;\; +\sum_{i=1}^{\mathrm{dim}\,\g}\mathcal{Y}_{M-n}(\gamma_i\cdot u_1\otimes [g,\gamma_i]\cdot u_2)+n\ell\mathcal{Y}_{M-n}(g\cdot u_1\otimes u_2).\nonumber\\
\end{align}
First consider the case $1\leq n\leq M-1$. Using \eqref{KZcomp} and \eqref{cu1u2orig}, the first term on the right of \eqref{calc1} becomes
\begin{align}\label{calc2}
 \mathcal{Y}_{M-n}\big([(\ell+h^\vee) & (h+M-n)-C_{U_1, U_2}](g\cdot u_1\otimes u_2)\big) =\mathcal{Y}_{M-n}\big((g\otimes 1) C^{(M)}_{U_1,U_2}(u_1\otimes u_2)\big)\nonumber\\
 & +\sum_{i=1}^{\mathrm{dim}\,\g} \mathcal{Y}_{M-n}([g,\gamma_i]\cdot u_1\otimes \gamma_i\cdot u_2)-n(\ell+h^\vee)\mathcal{Y}_{M-n}(g\cdot u_1\otimes u_2).
\end{align}
To analyze the second term on the right of \eqref{calc1}, we use another lemma:
\begin{lemma}\label{lemma2}
 In $U(\g)$, $\sum_{i=1}^{\mathrm{dim}\,\g} [g,\gamma_i]\gamma_i=h^\vee g$ for any $g\in\g$.
\end{lemma}
\begin{proof}
We know from Lemma \ref{lemma1} that $\sum_{i=1}^{\mathrm{dim}\,\g} [g,\gamma_i]\gamma_i=-\sum_{i=1}^{\mathrm{dim}\,\g} \gamma_i [g,\gamma_i]$. Therefore
\begin{align*}
 \sum_{i=1}^{\mathrm{dim}\,\g} [g,\gamma_i]\gamma_i=\dfrac{1}{2}\sum_{i=1}^{\mathrm{dim}\,\g} ([g,\gamma_i]\gamma_i-\gamma_i[g,\gamma_i])=\dfrac{1}{2}\sum_{i=1}^{\mathrm{dim}\,\g} [[g,\gamma_i],\gamma_i]= h^\vee g,
\end{align*}
recalling that the Casimir operator on the adjoint representation $\g$ is the scalar $2h^\vee$.
\end{proof}
Now we insert \eqref{calc2} back into \eqref{calc1} and cancel terms using Lemmas \ref{lemma1} and \ref{lemma2}. Only the first term to the right of the equality in \eqref{calc2} survives, completing the proof of the case $1\leq n\leq M-1$.

% 
% Using \eqref{calc2}, Lemma \ref{lemma1}, and Lemma \ref{lemma2}, the right side of \eqref{calc1} now simplifies to 
% \begin{equation*}
% \mathcal{Y}_{M-n}((g\otimes 1) C^{(M)}_{\lambda_1,\lambda_2}(u\otimes v)),
% \end{equation*}
% completing the proof of the case $1\leq n\leq M-1$.

Finally for the case $n=M$, the first term on the right of \eqref{calc1} vanishes. We calculate the remaining terms using Lemmas \ref{lemma1} and \ref{lemma2}, and then \eqref{cu1u2orig} and \eqref{holdsfor0}: 
\begin{align*}
 M(\ell+h^\vee) & \mathcal{Y}_0(g\cdot u_1\otimes u_2)-\sum_{i=1}^{\mathrm{dim}\,\g}\mathcal{Y}_0([g,\gamma_i]\cdot u_1\otimes\gamma_i\cdot u_2)\nonumber\\
 & =\mathcal{Y}_0\big((g\otimes 1)[(\ell+h^\vee)M-C_{U_1,U_2}](u_1\otimes u_2)\big)+\mathcal{Y}_0\big(C_{U_1,U_2}(g\cdot u_1\otimes u_2)\big)\nonumber\\
 & =\mathcal{Y}_0\big((g\otimes 1) C^{(M)}_{U_1, U_2}(u_1\otimes u_2)\big)+\mathcal{Y}_0\big([C_{U_1,U_2}-(\ell+h^\vee)h](g\cdot u_1\otimes u_2)\big)\nonumber\\
 & =\mathcal{Y}_0\big((g\otimes 1) C^{(M)}_{U_1,U_2}(u_1\otimes u_2)).
\end{align*}
This completes the proof of the theorem.
\end{proof}

Now our main theorem combines Theorems \ref{intwopext} and \ref{maintheorem1}:
\begin{theo}\label{maintheorem2}
 Suppose $U_1$, $U_2$ are $\g$-modules and $W_3$ is an $\N$-gradable weak $\gvmzero$-module such that $L(0)$ acts on $U_1$, $U_2$, and $W_3(0)$ by scalars $h_1$, $h_2$, and $h_3$, respectively. If moreover $(\ell+h^\vee)(h_3+N)-\frac{1}{2} C_{U_1\otimes U_2}$ is invertible on $U_1\otimes U_2$ for all $N\in\Z_+$, then the linear map
 \begin{align*}
  \mathcal{V}^{W_3}_{\gvmone\,\gvmtwo} & \rightarrow\mathrm{Hom}_\g(U_1\otimes U_2, W_3(0))\nonumber\\
  \mathcal{Y} & \mapsto \mathcal{Y}_0\vert_{U_1\otimes U_2}
 \end{align*}
is an isomorphism.
\end{theo}
\begin{proof}
 By \eqref{eqn:Casimir_relation}, $(\ell+h^\vee)(h+N)-C_{U_1,U_2}$ is invertible if $(\ell+h^\vee)(h_3+N)-\frac{1}{2} C_{U_1\otimes U_2}$ is invertible. Then Theorems \ref{intwopext} and \ref{maintheorem1} imply that given $f\in\mathrm{Hom}_\g(U_1\otimes U_2, W_3(0))$, there is a unique intertwining operator $\mathcal{Y}$ of type $\binom{W_3}{\gvmone\,\gvmtwo}$ such that $\mathcal{Y}_0\vert_{U_1\otimes U_2} =f$. That is, $\mathcal{Y}\mapsto\mathcal{Y}_0\vert_{U_1\otimes U_2}$ is both injective and surjective.
\end{proof}

\section{Analysis of the obstructions}\label{sec:obstructions}

In this section, we discuss whether non-invertibility of $(\ell+h^\vee)(h_3+N)-\frac{1}{2}C_{U_1\otimes U_2}$ for some $N\in\Z_+$, as in Theorem \ref{maintheorem2}, is truly an obstruction to constructing intertwining operators. For simplicity, we will assume that $U_1$, $U_2$, and $U_3=W_3(0)$ are finite-dimensional irreducible $\g$-modules corresponding to dominant integral weights $\lambda_1$, $\lambda_2$, and $\lambda_3$, respectively; this will guarantee that $C_{U_1\otimes U_2}$ is diagonalizable. We will focus on the cases that $W_3$ is a generalized Verma module or its contragredient dual (see \cite[Section 5.2]{FHL}). In this setting, Theorem \ref{maintheorem2} reads:
\begin{theo}\label{thm:main_lambda_theo}
 Suppose $\lambda_1$, $\lambda_2$, and $\lambda_3$ are dominant integral weights of $\g$ and $W_3$ is an $\N$-gradable weak $\gvmzero$-module such that $W_3(0)=L_{\lambda_3}$. If moreover $2(\ell+h^\vee)(h_{\lambda_3,\ell}+N)$ is not an eigenvalue of $C_{L_{\lambda_1}\otimes L_{\lambda_2}}$ for any $N\in\Z_+$, then there is a linear isomorphism
 \begin{align*}
  \mathcal{V}^{W_3}_{V_\g(\ell,\lambda_1)\,V_\g(\ell,\lambda_2)} & \rightarrow\mathrm{Hom}_\g(L_{\lambda_1}\otimes L_{\lambda_2}, L_{\lambda_3})\nonumber\\
  \mathcal{Y} & \mapsto \mathcal{Y}_0\vert_{L_{\lambda_1}\otimes L_{\lambda_2}}
 \end{align*}
\end{theo}

The easiest case to analyze is that $W_3$ is the contragredient $V_\g(\ell,\lambda_3^*)'$ of a generalized Verma module, where $\lambda_3^*$ is the dominant integral weight of $\g$ such that $L_{\lambda_3^*}\cong L_{\lambda_3}^*$. Then there are no obstructions: $\mathcal{V}^{V_\g(\ell,\lambda_3^*)'}_{V_\g(\ell,\lambda_1)\,V_\g(\ell,\lambda_2)} \cong\mathrm{Hom}_\g(L_{\lambda_1}\otimes L_{\lambda_2}, L_{\lambda_3})$ unconditionally. This follows from \cite[Theorem 2.11]{Li2}, or can be proved using Theorem \ref{intwopext}. To produce the maps $\mathcal{Y}_m: L_{\lambda_1}\otimes L_{\lambda_2}\rightarrow V_\g(\ell,\lambda_3^*)'(m)$ satisfying \eqref{commcomp} required by Theorem \ref{intwopext}, one starts with $\mathcal{Y}_0=f$ for any $f\in\mathrm{Hom}_\g(L_{\lambda_1}\otimes L_{\lambda_2}, L_{\lambda_3})$ and then recursively defines $\mathcal{Y}_m$ by
\begin{equation*}
 \langle \mathcal{Y}_m(u_1\otimes u_2), g(-n)w_3\rangle = -\langle \mathcal{Y}_{m-n}(g\cdot u_1\otimes u_2), w_3\rangle
\end{equation*}
for $u_1\in L_{\lambda_1}$, $u_2\in L_{\lambda_2}$, $g\in\g$, $1\leq n\leq m$, and $w_3\in V_\g(\ell,\lambda_3^*)(m-n)$.

Although we do not need Theorems \ref{intwopext} and \ref{maintheorem1} to determine $\mathcal{V}^{V_\g(\ell,\lambda_3^*)'}_{V_\g(\ell,\lambda_1)\,V_\g(\ell,\lambda_2)}$, they still provide information. If the conditions of Theorem \ref{thm:main_lambda_theo} hold, the construction given by Theorems \ref{intwopext} and \ref{maintheorem1} shows that the image of every intertwining operator of type $\binom{V_\g(\ell,\lambda_3^*)'}{V_\g(\ell,\lambda_1)\,V_\g(\ell,\lambda_2)}$ is contained in the $\gvmzero$-submodule of $V_\g(\ell,\lambda_3^*)'$ generated by the lowest weight space $V_\g(\ell,\lambda_3^*)'(0)=L_{\lambda_3^*}^*\cong L_{\lambda_3}$. This submodule is the radical of the unique maximal proper submodule $J_\g(\ell,\lambda_3^*)$ of $V_\g(\ell,\lambda_3^*)$ and thus is isomorphic to the contragredient $\ilambdathree$ of the irreducible quotient $L_\g(\ell,\lambda_3^*)=V_\g(\ell,\lambda_3^*)/J_\g(\ell,\lambda_3^*)$. So we have:
\begin{theo}
 Suppose $\lambda_1$, $\lambda_2$, and $\lambda_3$ are dominant integral weights of $\g$. If $2(\ell+h^\vee)(h_{\lambda_3,\ell}+N)$ is not an eigenvalue of $C_{L_{\lambda_1}\otimes L_{\lambda_2}}$ for any $N\in\Z_+$, then every intertwining operator of type $\binom{V_\g(\ell,\lambda_3^*)'}{V_\g(\ell,\lambda_1)\,V_\g(\ell,\lambda_2)}$ factors through the inclusion $\ilambdathree\hookrightarrow V_\g(\ell,\lambda_3^*)'$.
\end{theo}

The case that $W_3=V_\g(\ell,\lambda_3)$ is more interesting. Example \ref{exam:MY_sometimes_better} in the next section will show that the conditions of Theorem \ref{thm:main_lambda_theo} are not always necessary. However, \cite[Section 8]{MY} showed some examples of genuine obstructions to the existence of intertwining operators in the case $\g=\mathfrak{sl}_2$, arising from singular vectors in $V_\g(\ell,\lambda_3)$. We now show why singular vectors can be a problem when $C_{L_{\lambda_1}\otimes L_{\lambda_2}}$ has eigenvalue(s) $2(\ell+h^\vee)(h_{\lambda_3,\ell}+N)$:
\begin{theo}\label{thm:sing_vect_candidate}
 Suppose $\lambda_1$, $\lambda_2$, and $\lambda_3$ are dominant integral weights of $\g$ and $N$ is the smallest positive integer such that $2(\ell+h^\vee)(h_{\lambda_3,\ell}+N)$ is an eigenvalue of $C_{L_{\lambda_1}\otimes L_{\lambda_2}}$. Given $f\in\mathrm{Hom}_\g(L_{\lambda_1}\otimes L_{\lambda_2}, L_{\lambda_3})$, let
 \begin{equation*}
  \mathcal{Y}_m: L_{\lambda_1}\otimes L_{\lambda_2}\rightarrow V_\g(\ell,\lambda_3)(m)
 \end{equation*}
for $1\leq m\leq N-1$ denote the maps defined by the recursive formula \eqref{ym}, starting from $\mathcal{Y}_0=f$. Then if $\sum_j u^{(1)}_j\otimes u^{(2)}_j$ an eigenvector of $C_{L_{\lambda_1}\otimes L_{\lambda_2}}$ with eigenvalue $2(\ell+h^\vee)(h_{\lambda_{3,\ell}}+N)$,
\begin{equation}\label{sing_vect_candidate}
 \sum_{i=1}^{\dim\g}\sum_j \sum_{k=1}^{N} \gamma_i(-k)\mathcal{Y}_{N-k}(\gamma_i\cdot u^{(j)}_1\otimes u^{(j)}_2)
\end{equation}
lies in the maximal proper submodule $J_\g(\ell,\lambda_3)$. Moreover, if \eqref{sing_vect_candidate} is non-zero for some eigenvector, then there is no intertwining operator of type $\binom{V_\g(\ell,\lambda_3)}{V_\g(\ell,\lambda_1)\,V_\g(\ell,\lambda_2)}$ such that $\mathcal{Y}_0\vert_{L_{\lambda_1}\otimes L_{\lambda_2}}=f$.
\end{theo}
\begin{proof}
Because $N$ is the smallest positive integer such that $(\ell+h^\vee)(h+N)-C_{L_{\lambda_1},L_{\lambda_2}}$ is not invertible, \eqref{ym} is well defined for $1\leq m\leq N-1$. So if $W_3$ is any $\N$-gradable weak $\gvmzero$-module with $W_3(0)=L_{\lambda_3}$ and $\mathcal{Y}$ is an intertwining operator of type $\binom{W_3}{V_\g(\ell,\lambda_1)\,V_\g(\ell,\lambda_2)}$ such that $\mathcal{Y}_0\vert_{L_{\lambda_1}\otimes L_{\lambda_2}}=f$, then \eqref{KZcomp} implies that $\mathcal{Y}_m\vert_{L_{\lambda_1}\otimes L_{\lambda_2}}$ for $1\leq m\leq N-1$ must be given by \eqref{ym} and that \eqref{sing_vect_candidate} must vanish. The second assertion of the theorem then follows from taking $W_3=V_\g(\ell,\lambda_3)$.
 
 For the first assertion, take $W_3=V_\g(\ell,\lambda_3^*)'$. Then \eqref{sing_vect_candidate} is an expression in $U(\ghat_-)\otimes L_{\lambda_3}$ that vanishes in $V_\g(\ell,\lambda_3^*)'$ since in this case the intertwining operator $\mathcal{Y}$ exists. Since we have observed that the $\gvmzero$-submodule of $V_\g(\ell,\lambda_3^*)'$ generated by $L_{\lambda_3}$ is isomorphic to
 \begin{equation*}
  L(\ell,\lambda_3) = (U(\ghat_-)\otimes L_{\lambda_3})/J_\g(\ell,\lambda_3),
 \end{equation*}
it follows that \eqref{sing_vect_candidate} viewed as a vector in $V_\g(\ell,\lambda_3)$ lies in $J_\g(\ell,\lambda_3)$.
\end{proof}

\begin{rema}
Theorem \ref{thm:sing_vect_candidate} provides a recipe for producing candidates for singular vectors in $V_\g(\ell,\lambda_3)$. However, Example \ref{exam:MY_sometimes_better} will show that \eqref{sing_vect_candidate} can vanish even if $f\neq 0$. 
\end{rema}
\begin{rema}\label{rema:irr_level_no_ob}
 Since the eigenvalues of $C_{L_{\lambda_1}\otimes L_{\lambda_2}}$ and $C_{L_{\lambda_3}}$ are rational when the $\lambda_i$ are dominant integral weights of $\g$, obstructions to intertwining operators of type $\binom{V_\g(\ell,\lambda_3)}{V_\g(\ell,\lambda_1)\,V_\g(\ell,\lambda_2)}$ never arise if $\ell\notin\Q$. This is no surprise because in this case generalized Verma modules induced from finite-dimensional irreducible $\g$-modules are themselves irreducible and thus isomorphic to contragredients of generalized Verma modules. However, obstructions might occur for $\ell\notin\Q$ if the $\lambda_i$ are not dominant integral.
\end{rema}

\section{The case \texorpdfstring{$\g=\mathfrak{sl}_2$}{g=sl(2)} revisited}\label{sec:sl2}

We conclude by comparing Theorem \ref{thm:main_lambda_theo} in the case $\g=\mathfrak{sl}_2$ with the results of \cite{MY}, and by demonstrating some new examples of intertwining operators among generalized Verma modules for $\widehat{\mathfrak{sl}}_2$. For $p\in\N$, we let $L_p$ denote the $(p+1)$-dimensional irreducible $\mathfrak{sl}_2$-module of highest weight $p\frac{\alpha}{2}$ and let $\gvmp$ denote the corresponding generalized Verma module for $\widehat{\mathfrak{sl}}_2$. As mentioned in Remark \ref{rema:irr_level_no_ob}, we always have
\begin{equation*}
 \mathcal{N}^{\gvmr}_{\gvmp\,\gvmq} =\dim\mathrm{Hom}_{\mathfrak{sl}_2}(L_p\otimes L_q, L_r)
\end{equation*}
when $\ell\notin\Q$. For rational levels, we express Theorem \ref{thm:main_lambda_theo} as follows:
\begin{theo}\label{thm:sl2_theo}
 For $\ell\in\Q\setminus\lbrace-2\rbrace$ and $p,q,r\in\N$, the fusion rule $\mathcal{N}^{\gvmr}_{\gvmp\,\gvmq} =1$ under the following conditions:
 \begin{enumerate}
  \item $r=p+q-2n$ with $0\leq n\leq\min(p,q)$, and
  \item $m(m+r+1)\notin (\ell+2)\Z_+$ for $n-\min(p,q)\leq m\leq n$.
 \end{enumerate}
 Moreover, if $\ell+2\in\Q_{+}$, then we may replace condition (2) with the condition:
 \begin{enumerate}
  \item[(2')]  $m(m+r+1)\notin (\ell+2)\Z_+$ for $1\leq m\leq n$
 \end{enumerate}

\end{theo}
\begin{proof}
 The first condition guarantees that $\dim\mathrm{Hom}_{\mathfrak{sl}_2}(L_p\otimes L_q, L_r)=1$, so we just need to check that the second condition(s) guarantee that $2(\ell+h^\vee)(h_{r,\ell}+N)$ is not an eigenvalue of $C_{L_p\otimes L_q}$ for any $N\in\Z_+$. Recalling \eqref{hlambdal} and using $h^\vee=2$ for $\mathfrak{sl}_2$ and
 \begin{equation*}
  L_p\otimes L_q\cong\bigoplus_{k=0}^{\min(p,q)} L_{p+q-2k},
 \end{equation*}
the conditions of Theorem \ref{thm:main_lambda_theo} amount to
\begin{equation*}
 \frac{(p+q-2k)(p+q-2k+2)}{2} \neq 2(\ell+2)\left(\frac{(p+q-2n)(p+q-2n+2)}{4(\ell+2)}+N\right)
\end{equation*}
for $0\leq k\leq\min(p,q)$ and any $N\in\Z_+$. This simplifies to
\begin{equation}\label{eqn:sl2_theo_cond}
 (n-k)(p+q-n-k+1)=(n-k)(r+n-k+1)\notin(\ell+2)\Z_+.
\end{equation}
 Setting $m=n-k$, we get
\begin{equation*}
 m(m+r+1)\notin(\ell+2)\Z_+
\end{equation*}
for $n-\min(p,q)\leq m\leq n$, as desired. If $\ell+2\in\Q_+$, then the right side of \eqref{eqn:sl2_theo_cond} is strictly positive while the left side is non-positive unless $0\leq k\leq n-1$, so we may restrict to $1\leq m\leq n$.
\end{proof}
\begin{exam}\label{exam:sl2_theo}
 For $\ell+2\in\Q_+$, we determine when obstructions to intertwining operators can possibly occur in the cases $n=0,1,2$ of Theorem \ref{thm:sl2_theo}:
 \begin{itemize}
  \item For $p,q\in\N$, $\mathcal{N}^{V_{\mathfrak{sl}_2}(\ell,p+q)}_{\gvmp\,\gvmq} =1$ unconditionally.
  \item For $p,q\in\N$ such that $\min(p,q)\geq 1$, $\mathcal{N}^{V_{\mathfrak{sl}_2}(\ell,p+q-2)}_{\gvmp\,\gvmq} =1$ except possibly when $p+q\in(\ell+2)\Z_+$.
  \item For $p,q\in\N$ such that $\min(p,q)\geq 2$, $\mathcal{N}^{V_{\mathfrak{sl}_2}(\ell,p+q-4)}_{\gvmp\,\gvmq} =1$ except possibly when $2(p+q-1)\in(\ell+2)\Z_+$ or $p+q-2\in(\ell+2)\Z_+$.
 \end{itemize}
\end{exam}

Theorem \ref{thm:sl2_theo} is similar to but somewhat different from the intertwining operator theorems of \cite{MY}. In \cite[Theorem 6.1]{MY}, we assumed that the maximal proper submodule of $\gvmr$ was irreducible and isomorphic to some $L_{\mathfrak{sl}_2}(\ell,r')$, and that the maximal proper submodule of $V_{\mathfrak{sl}_2}(\ell,r')$ was irreducible and isomorphic to some $L_{\mathfrak{sl}_2}(\ell,r'')$. Under these assumptions, we proved that $\mathcal{V}^{\gvmr}_{\gvmp\,\gvmq}\cong\mathrm{Hom}_{\mathfrak{sl}_2}(L_p\otimes L_q, L_r)$ provided
\begin{equation*}
 \mathrm{Hom}_{\mathfrak{sl}_2}(L_p\otimes L_q, L_{r'})=\mathrm{Hom}_{\mathfrak{sl}_2}(L_p\otimes L_q, L_{r''})=0,
\end{equation*}
that is, provided $2(\ell+2)h_{r',\ell}$ and $2(\ell+2)h_{r'',\ell}$ are not eigenvalues of $C_{L_p\otimes L_q}$. Now, since $h_{r',\ell}$ is the lowest conformal weight of a proper non-zero submodule of $\gvmr$, we must have $h_{r',\ell}=h_{r,\ell}+N'$ for some $N'\in\Z_+$, and then similarly $h_{r'',\ell}=h_{r,\ell}+N''$ for some $N''$. Thus \cite[Theorem 6.1]{MY} implies Theorem \ref{thm:sl2_theo}, but only for $\gvmr$ that satisfy the irreducibility assumptions on maximal proper submodules.

In practice, we know from \cite[Theorem 3.8]{MY} that $J_{\mathfrak{sl}_2}(\ell,r)\cong L_{\mathfrak{sl}_2}(\ell,r')$ and $J_{\mathfrak{sl}_2}(\ell,r')\cong L_{\mathfrak{sl}_2}(\ell,r'')$ when $\ell\in\N$ and $\gvmr$ appears in the Garland-Lepowsky resolutions \cite{GL} of integrable highest-weight $\widehat{\mathfrak{sl}}_2$-modules. Specifically, this means
\begin{equation}\label{eqn:mjn_def}
 r=m(j,n)=(\ell+2)j+\frac{\ell}{2}(1-(-1)^j)+(-1)^j n
\end{equation}
with $j\geq 0$ and $0\leq n\leq\ell$ (see \cite[Proposition 8.2]{MY}), and then $r'=m(j+1,n)$, $r''=m(j+2,n)$. For such $r$, the next example shows that \cite[Theorem 6.1]{MY} can be stronger than Theorem \ref{thm:sl2_theo}, that is, the eigenvalue conditions of Theorem \ref{thm:main_lambda_theo} are not always necessary for generalized Verma modules.
\begin{exam}\label{exam:MY_sometimes_better}
 For $\ell\in 2\Z_+$, take the $\mathfrak{sl}_2$-homomorphism $L_{\ell/2}\otimes L_{\ell/2}\rightarrow L_0\cong\C$ given by a non-zero invariant bilinear form $\langle\cdot,\cdot\rangle_{\ell/2}$. Then $r=0$, $r'=2\ell+2$, and $r''=2\ell+4$, so
 \begin{equation*}
  \mathrm{Hom}_{\mathfrak{sl}_2}(L_{\ell/2}\otimes L_{\ell/2}, L_{r'})=\mathrm{Hom}_{\mathfrak{sl}_2}(L_{\ell/2}\otimes L_{\ell/2}, L_{r''})=0
 \end{equation*}
and by \cite[Theorem 6.1]{MY} there is a unique intertwining operator $\mathcal{Y}$ of type $\binom{V_{\mathfrak{sl}_2}(\ell,0)}{V_{\mathfrak{sl}_2}(\ell,\ell/2)\,V_{\mathfrak{sl}_2}(\ell,\ell/2)}$ such that for $u_1, u_2\in L_{\ell/2}$,
\begin{equation*}
 \mathcal{Y}_0(u_1\otimes u_2)=\langle u_1,u_2\rangle_{\ell/2}.
\end{equation*}
But condition (2') of Theorem \ref{thm:sl2_theo} fails when $\ell\in 4\Z_+$: for $m=n=\ell/2$, we have
\begin{equation*}
 m(m+r+1)=\frac{\ell}{2}\left(\frac{\ell}{2}+0+1\right)=(\ell+2)\frac{\ell}{4}\in(\ell+2)\Z_+.
\end{equation*}
This means that $2(\ell+2)\left(h_{0,\ell}+\frac{\ell}{4}\right)$ is an eigenvalue of $C_{L_{\ell/2}\otimes L_{\ell/2}}$ (note that $h_{0,\ell}=0$). 

When $\ell=4$, it is easy to check that the candidate \eqref{sing_vect_candidate} for a singular vector in $V_{\mathfrak{sl}_2}(\ell,0)$ vanishes. Here, $L_{\ell/2}\cong\mathfrak{sl}_2$ with standard basis $\lbrace e,h,f\rbrace$, we scale $\langle\cdot,\cdot\rangle_{\ell/2}=\langle\cdot,\cdot\rangle$ so that $\langle h,h\rangle =2$ and $\langle e,f\rangle=\langle f,e\rangle=1$, and the appropriately-scaled Casimir operator is $ef+\frac{1}{2}h^2+fe$. The most interesting eigenvector of $C_{\mathfrak{sl}_2\otimes\mathfrak{sl}_2}$ with eigenvalue $2(\ell+2)\left(0+\frac{\ell}{4}\right)=12$ to check is $e\otimes h+h\otimes e$: then \eqref{sing_vect_candidate} becomes
\begin{align*}
 e(-1) & \left(\langle [f,h],e\rangle+\langle h,[f,e]\rangle\right)+\frac{1}{2} h(-1)\left(\langle [h,h],e\rangle+\langle [h,e],h\rangle\right)\nonumber\\ &\hspace{6em}+f(-1)\left(\langle [e,h],e\rangle+\langle[e,e],h\rangle\right)\nonumber\\
 & = e(-1)(2\langle f,e\rangle-\langle h,h\rangle) +h(-1)\langle e,h\rangle -2 f(-1)\langle e,e\rangle =0.
\end{align*}
\end{exam}

Although this example shows that \cite{MY} can give better results than Theorem \ref{thm:sl2_theo}, Theorem \ref{thm:sl2_theo} is usually more versatile. Especially, Theorem \ref{thm:sl2_theo} applies to any level $\ell\in\mathbb{Q}\setminus\lbrace -2\rbrace$ and any weight $r\in\N$. For example, we can take $r$ for which $\gvmr$ does not appear in the Garland-Lepowsky resolutions: from \eqref{eqn:mjn_def}, these are the positive integers $(\ell+2)j-1$ with $j\geq 1$. Then Example \ref{exam:sl2_theo} provides for instance many new examples of non-zero intertwining operators of type $\binom{V_{\mathfrak{sl}_2}(\ell,(\ell+2)j-1)}{\gvmp\,\gvmq}$ where $p+q=(\ell+2)j-1$. Moreover, when the conditions of Theorem \ref{thm:sl2_theo} fail, we can use Theorem \ref{thm:sing_vect_candidate} to compute candidates for singular vectors in $V_{\mathfrak{sl}_2}(\ell,(\ell+2)j-1)$. For example $\ell=0$, $p=2$, $q=3$, $r=1$ yields a singular vector candidate of conformal weight $\frac{35}{8}$ in $V_{\mathfrak{sl}_2}(0,1)$. It would be interesting to check if \eqref{sing_vect_candidate} is non-zero in this case, although the calculations would be involved.

Another advantage of Theorem \ref{thm:sl2_theo}, or rather Theorem \ref{maintheorem2}, is that we can use it to construct intertwining operators among generalized Verma modules induced from irreducible infinite-dimensional $\mathfrak{g}$-modules $U_1$, $U_2$, and $U_3$, at least as long as $C_{U_1\otimes U_2}$ is diagonalizable. As discussed in the Introduction, understanding such intertwining operators for affine Lie algebras at admissible rational levels is of considerable interest. We conclude this paper by discussing some extensions of Theorem \ref{thm:sl2_theo} to infinite-dimensional $\mathfrak{sl}_2$-modules.

For $\lambda\in\mathbb{C}$, let $L_\lambda$ denote the irreducible highest-weight $\mathfrak{sl}_2$-module with highest weight $\lambda\frac{\alpha}{2}$. When $\lambda\in\C\setminus\N$, $L_\lambda$ is also a Verma module. If $p\in\N$ and $\lambda,p+\lambda\in\C\setminus\N$, we have
\begin{equation*}
 L_p\otimes L_\lambda =\bigoplus_{k=0}^p L_{p+\lambda-2k};
\end{equation*}
this means $C_{L_p\otimes L_\lambda}$ is diagonalizable with eigenvalues
\begin{equation*}
 \frac{1}{2}(p+\lambda-2k)(p+\lambda-2k+2),\qquad 0\leq k\leq p.
\end{equation*}
Similarly, if $\lambda_1,\lambda_2,\lambda_1+\lambda_2\in\C\setminus\N$, then
\begin{equation*}
 L_{\lambda_1}\otimes L_{\lambda_2} =\bigoplus_{k=0}^{\infty} L_{\lambda_1+\lambda_2-2k},
\end{equation*}
so $C_{L_{\lambda_1}\otimes L_{\lambda_2}}$ is diagonalizable with eigenvalues
\begin{equation*}
 \frac{1}{2}(\lambda_1+\lambda_2-2k)(\lambda_1+\lambda_2-2k+2),\qquad k\in\N.
\end{equation*}
In both these cases, the calculation in the proof of Theorem \ref{thm:sl2_theo} leads to the same condition \eqref{eqn:sl2_theo_cond}. So we get the following theorem:
\begin{theo}\label{thm:inf_sl2_theo}
 If $p\in\N$ and $\lambda,p+\lambda\in\C\setminus\N$, then the fusion rule $\mathcal{N}_{V_{\mathfrak{sl}_2}(\ell,p)\,V_{\mathfrak{sl}_2}(\ell,\lambda)}^{V_{\mathfrak{sl}_2}(\ell,\mu)} =1$ under the following conditions:
 \begin{enumerate}
  \item $\mu=p+\lambda-2n$ with $0\leq n\leq p$, and
  \item $m(m+\mu+1)\notin(\ell+2)\Z_+$ for $m=n, n-1,\ldots, n-p$.
 \end{enumerate}
If $\lambda_1,\lambda_2,\lambda_1+\lambda_2\in\C\setminus\N$, then the fusion rule $\mathcal{N}_{V_{\mathfrak{sl}_2}(\ell,\lambda_1)\,V_{\mathfrak{sl}_2}(\ell,\lambda_2)}^{V_{\mathfrak{sl}_2}(\ell,\lambda_3)} =1$ under the following conditions:
 \begin{enumerate}
  \item $\lambda_3=\lambda_1+\lambda_2-2n$ with $n\in\N$, and
  \item $m(m+\lambda_3+1)\notin(\ell+2)\Z_+$ for $m =n,n-1,n-2,\ldots$.
 \end{enumerate}
\end{theo}

Note that it will be easy for condition (2) in the second assertion of the theorem to fail if $\lambda_3$ and  $\ell$ are both rational, since there is no upper bound for the expression $m(m+\lambda_3+1)$. However, the first assertion of the theorem is easy to apply in many cases. As an example, we will look at an admissible level for $\mathfrak{sl}_2$.  Admissible levels for $\mathfrak{sl}_2$ have the form $\ell=-2+\frac{u}{v}$ where $u\geq 2$ and $v\geq 1$ are relatively prime integers. For $\ell$ an admissible level, irreducible highest-weight modules for the simple affine vertex operator algebra $L_{\mathfrak{sl}_2}(\ell,0)$ were classified in \cite{AM, DLM}; they are the simple modules $L_{\mathfrak{sl}_2}(\ell,\lambda_{r,s})$ where 
 \begin{equation*}
  \lambda_{r,s}= r-1-(\ell+2)s,
 \end{equation*}
for $1\leq r\leq u-1$ and $0\leq s\leq v-1$. Fusion rules involving these irreducible highest-weight modules $L_{\mathfrak{sl}_2}(\ell,\lambda_{r,s})$ are known \cite{BF, DLM}. Here we look at intertwining operators involving generalized Verma modules at one of the simplest non-integral admissible levels, $\ell=-2+\frac{3}{2}=-\frac{1}{2}$:
\begin{exam}\label{exam:adm_hw}
For $\ell=-\frac{1}{2}$, the admissible $\mathfrak{sl}_2$-weights $\lambda_{r,s}$ are given by
\begin{equation*}
 \lambda=0,1,-\frac{3}{2},-\frac{1}{2}.
\end{equation*}
Looking at homomorphisms involving irreducible $\mathfrak{sl}_2$-modules with these highest weights, we see that there might be non-zero intertwining operators of types $\binom{V_{\mathfrak{sl}_2}(-\frac{1}{2},-\frac{1}{2})}{V_{\mathfrak{sl}_2}(-\frac{1}{2},1)\,V_{\mathfrak{sl}_2}(-\frac{1}{2},-\frac{3}{2})}$ and $\binom{V_{\mathfrak{sl}_2}(-\frac{1}{2},-\frac{3}{2})}{V_{\mathfrak{sl}_2}(-\frac{1}{2},1)\,V_{\mathfrak{sl}_2}(-\frac{1}{2},-\frac{1}{2})}$. These are the $\ell=-\frac{1}{2}$, $p=1$, $\lambda=-\frac{3}{2}$, $n=0$ and $\ell=-\frac{1}{2}$, $p=1$, $\lambda=-\frac{1}{2}$, $n=1$ cases of Theorem \ref{thm:inf_sl2_theo}. For the first case, we check
\begin{align*}
 0\left(0-\frac{1}{2}+1\right) = 0\notin\frac{3}{2}\Z_+,\qquad
 -1\left(-1-\frac{1}{2}+1\right) =\frac{1}{2}\notin\frac{3}{2}\Z_+,
\end{align*}
and for the second case, we check
\begin{equation*}
 1\left(1-\frac{3}{2}+1\right) =\frac{1}{2}\notin\frac{3}{2}\Z_+,\qquad 0\left(0-\frac{3}{2}+1\right) =0\notin\frac{3}{2}\Z_+.
\end{equation*}
So Theorem \ref{thm:inf_sl2_theo} implies that indeed $\mathcal{N}_{V_{\mathfrak{sl}_2}(-1/2,1)\,V_{\mathfrak{sl}_2}(-1/2,-3/2)}^{V_{\mathfrak{sl}_2}(-1/2,-1/2)} =\mathcal{N}_{V_{\mathfrak{sl}_2}(-1/2,1)\,V_{\mathfrak{sl}_2}(-1/2,-1/2)}^{V_{\mathfrak{sl}_2}(-1/2,-3/2)}=1$.

% Indeed, \cite[Section 5]{Ri} shows that if tensor products of these highest-weight $\mathfrak{sl}_2$-modules exist and are non-zero, then we should have
% \begin{equation*}
%  L_{\mathfrak{sl}_2}(-1/2,1)\boxtimes L_{\mathfrak{sl}_2}(-1/2,-3/2) = L_{\mathfrak{sl}_2}(-1/2,-1/2)
% \end{equation*}
% and
% \begin{equation*}
%  L_{\mathfrak{sl}_2}(-1/2,1)\boxtimes L_{\mathfrak{sl}_2}(-1/2,-1/2)=L_{\mathfrak{sl}_2}(-1/2,-3/2).
% \end{equation*}
% Thus we would expect non-zero intertwining operators of the indicated types to exist.
% 
% 
% 
% Now, Theorem \ref{thm:inf_sl2_theo} can be used to show that at least there are non-zero intertwining operator of types $\binom{L_{\mathfrak{sl}_2}(-\frac{1}{2},-\frac{1}{2})}{V_{\mathfrak{sl}_2}(-\frac{1}{2},1)\,V_{\mathfrak{sl}_2}(-\frac{1}{2},-\frac{3}{2})}$ and $\binom{L_{\mathfrak{sl}_2}(-\frac{1}{2},-\frac{3}{2})}{V_{\mathfrak{sl}_2}(-\frac{1}{2},1)\,V_{\mathfrak{sl}_2}(-\frac{1}{2},-\frac{1}{2})}$.

We can now project onto irreducible quotient modules to get non-zero intertwining operators
\begin{equation*}
 \mathcal{Y}: V_{\mathfrak{sl}_2}(-1/2,1)\otimes V_{\mathfrak{sl}_2}(-1/2,-3/2)\rightarrow L_{\mathfrak{sl}_2}(-1/2,-1/2)\lbrace x\rbrace
\end{equation*}
and
\begin{equation*}
 \mathcal{Y}: V_{\mathfrak{sl}_2}(-1/2,1)\otimes V_{\mathfrak{sl}_2}(-1/2,-1/2)\rightarrow L_{\mathfrak{sl}_2}(-1/2,-3/2)\lbrace x\rbrace.
\end{equation*}
 As suggested in the Introduction, we could recover the known intertwining operators among irreducible highest-weight $\widehat{\mathfrak{sl}}_2$-modules if we could show that the above intertwining operators remain well defined on the irreducible quotients of the generalized Verma modules. This would require information about the singular vectors generating the maximal proper submodules of the generalized Verma modules (and such information is what was used for the fusion rule calculations in \cite{DLM}).
\end{exam}

So far, we have considered $\widehat{\mathfrak{sl}}_2$-modules that are induced from highest-weight irreducible $\mathfrak{sl}_2$-modules. But relaxed highest-weight modules are also important: these are (quotients of) generalized Verma modules induced from more general irreducible weight $\mathfrak{sl}_2$-modules. For $\mathfrak{sl}_2$, these additional weight modules include lowest-weight modules and the irreducible modules $E_{\overline{\lambda},\delta}$ for certain $\overline{\lambda}=\lambda+2\Z\in\C/2\Z$ and $\delta\in\C$. The $\mathfrak{sl}_2$-module $E_{\overline{\lambda},\delta}$ has a one-dimensional weight space of weight $\lambda\frac{\alpha}{2}$ for each $\lambda\in\overline{\lambda}$, and $\delta$ is the eigenvalue of the Casimir operator $ef+\frac{1}{2}h^2+fe$. That is, for $v_\lambda\in E_{\overline{\lambda},\delta}$ of weight $\lambda\frac{\alpha}{2}$, 
\begin{equation*}
 ef\cdot v_\lambda =\frac{1}{2}\left(\delta-\frac{1}{2}\lambda(\lambda-2)\right).
\end{equation*}
In particular, we must have $\delta\neq\frac{1}{2}\lambda(\lambda-2)$ for all $\lambda\in\overline{\lambda}$ to avoid highest- or lowest-weight vectors in $E_{\overline{\lambda},\delta}$. In our last example, we again consider the admissible level $\ell=-\frac{1}{2}$:
\begin{exam}
 We refer to \cite[Section 2]{CR1} for more details on fusion rules for irreducible relaxed highest-weight $\widehat{\mathfrak{sl}}_2$-modules at $\ell=-\frac{1}{2}$. The irreducible $\widehat{\mathfrak{sl}}_2$-module $L_{\mathfrak{sl}_2}(-\frac{1}{2}, E_{\overline{\lambda},\delta})$ is an $L_{\mathfrak{sl}_2}(-\frac{1}{2},0)$-module for $\delta=-\frac{3}{8}$ and $\overline{\lambda}\neq\pm\frac{1}{2}+2\Z$. According to \cite[Equation (2.11)]{CR1}, there is a one-dimensional space of intertwining operators of type $\binom{L_{\mathfrak{sl}_2}(-\frac{1}{2}, E_{\overline{\lambda+1},-3/8})}{L_{\mathfrak{sl}_2}(-\frac{1}{2}, 1)\,L_{\mathfrak{sl}_2}(-\frac{1}{2}, E_{\overline{\lambda},-3/8})}$ for all such $\lambda$. We shall show that the same result holds for the corresponding generalized Verma modules for $V_{\mathfrak{sl}_2}(-\frac{1}{2},0)$.
 
 Note first that $L_1\otimes E_{\overline{\lambda},\delta}$ is a weight $\mathfrak{sl}_2$-module with two-dimensional weight spaces; the $\mathfrak{sl}_2$-weights are of the form $(\lambda+1)\frac{\alpha}{2}$ for $\lambda\in\overline{\lambda}$. A calculation shows that the Casimir operator on $L_1\otimes E_{\overline{\lambda},\delta}$ has eigenvalue(s)
 \begin{equation*}
  \delta_\pm =\frac{1}{2}(2\delta+1)\pm\sqrt{2\delta+1}
 \end{equation*}
on each weight space, independently of $\lambda\in\overline{\lambda}$. Thus for $\delta=-\frac{3}{8}$ the Casimir is diagonalizable with eigenvalues $-\frac{3}{8}, \frac{5}{8}$. As the $-\frac{3}{8}$-eigenspace is a direct summand of $L_1\otimes E_{\overline{\lambda},-3/8}$ which must be isomorphic to $E_{\overline{\lambda+1},-3/8}$, we see that
\begin{equation*}
\dim\mathrm{Hom}_{\mathfrak{sl}_2}( L_1\otimes E_{\overline{\lambda},-3/8}, E_{\overline{\lambda+1},-3/8})=1.
\end{equation*}
Thus to apply Theorem \ref{maintheorem2}, we just need to show that $\delta+2(\ell+h^\vee)N$ is not an eigenvalue of the Casimir operator for $N\in\Z_+$. In fact,
\begin{equation*}
 \delta+2(\ell+h^\vee)N=-\frac{3}{8}+2\left(-\frac{1}{2}+2\right)N=-\frac{3}{8}+3N\neq -\frac{3}{8},\frac{5}{8}
\end{equation*}
when $N\in\Z_+$, so we conclude that
\begin{equation*}
 \mathcal{N}_{V_{\mathfrak{sl}_2}(-\frac{1}{2},1)\,V_{\mathfrak{sl}_2}(-\frac{1}{2},E_{\overline{\lambda},-3/8})}^{V_{\mathfrak{sl}_2}(-\frac{1}{2},E_{\overline{\lambda+1},-3/8})} =1
\end{equation*}
for all $\overline{\lambda}\neq\pm\frac{1}{2}+2\Z$.

Just as in Example \ref{exam:adm_hw}, we can also get a one-dimensional space of intertwining operators
\begin{equation*}
 \mathcal{Y}: V_{\mathfrak{sl}_2}(-1/2,1)\otimes V_{\mathfrak{sl}_2}(-1/2,E_{\overline{\lambda},-3/8})\rightarrow L_{\mathfrak{sl}_2}(-1/2,E_{\overline{\lambda+1},-3/8})\lbrace x\rbrace.
\end{equation*}
Showing that these descend to well-defined intertwining operators among three irreducible $L_{\mathfrak{sl}_2}(-\frac{1}{2},0)$-modules would again require information about singular vectors.
\end{exam}

\section*{Acknowledgements}

I would like to thank Jinwei Yang for introducing me to the problem of constructing intertwining operators among generalized Verma modules. I would also like to thank Thomas Creutzig, Simon Wood, and the referees for comments, corrections, and suggestions.


\begin{thebibliography}{DLM}

\bibitem[Ad]{Ad}
D. Adamovi\'{c}, Realizations of simple affine vertex algebras and their modules: the cases $\widehat{\mathfrak{sl}(2)}$ and $\widehat{\mathfrak{osp}(1,2)}$, \textit{Comm. Math. Phys.} \textbf{366} (2019), 1025--1067.

\bibitem[AM]{AM}
D. Adamovi\'{c} and A. Milas, Vertex operator algebras associated to modular invariant representations for $A_1^{(1)}$, \textit{Math. Res. Lett.} \textbf{2} (1995), 563--575.

\bibitem[AP]{AP}
D. Adamovi\'{c} and V. Pedi\'{c}, On fusion rules and intertwining operators for the Weyl vertex algebra, \textit{J. Math. Phys.} \textbf{60} (2019), 081701, 18 pp.

\bibitem[BF]{BF}
D. Bernard and G. Felder, Fock representations and BRST cohomology in $SL(2)$ current algebra, \textit{Comm. Math. Phys.} \textbf{127} (1990), 145--168.  

\bibitem[CR1]{CR1}
T. Creutzig and D. Ridout, Modular data and Verlinde formulae for fractional level WZW models I, \textit{Nuclear Phys. B} \textbf{865} (2012), 83--114.

\bibitem[CR2]{CR2}
T. Creutzig and D. Ridout, Modular data and Verlinde formulae for fractional level WZW models II, \textit{Nuclear Phys. B} \textbf{875} (2013), 423--458.

\bibitem[DLM]{DLM}
C. Dong, H. Li and G. Mason, Vertex operator algebras associated to admissible representations of $\widehat{\mathfrak{sl}}_2$, \textit{Comm. Math. Phys.} \textbf{184} (1997), 65--93. 

\bibitem[FHL]{FHL}
I. Frenkel, Y.-Z. Huang, and J. Lepowsky, On axiomatic
approaches to vertex operator algebras and modules, \textit{Mem. Amer.
Math. Soc.} \textbf{104} (1993).

\bibitem[FZ]{FZ}
I. Frenkel and Y.  Zhu, Vertex operator algebras associated to
representations of affine and Virasoro algebras, {\em Duke Math. J.}
{\bf 66} (1992),  123--168.

\bibitem[GL]{GL}
H. Garland, J. Lepowsky, Lie algebra homology and
the Macdonald-Kac formulas. {\em Invent. Math.} {\bf 34} (1976),
 37--76.

\bibitem[HL1]{HL1}
Y.-Z. Huang and J. Lepowsky, A theory of tensor products for module categories for a vertex operator algebra, I, \textit{Selecta Math. (N. S.)} \textbf{1} (1995), 699--756.
 
\bibitem[HL2]{HL2}
Y.-Z. Huang and J. Lepowsky, Tensor categories and the mathematics of logarithmic and rational conformal field theory, \textit{J. Phys. A: Math. Theor.} \textbf{46} 494009 (2013).

\bibitem[HLZ]{HLZ2}
Y.-Z. Huang, J. Lepowsky and L. Zhang, Logarithmic tensor category theory for 
generalized modules for a conformal vertex algebra, II: Logarithmic formal 
calculus and properties of logarithmic intertwining operators, arXiv:1012.4196.

\bibitem[Hu]{Hu}
J. Humphreys, {\em Introduction to Lie Algebras and Representation
Theory}, Springer Verlag, New York, 1972.

\bibitem[KZ]{KZ}
V. Knizhnik and A. Zamolodchikov, Current algebra and Wess-Zumino models in two dimensions, \textit{Nuclear Phys. B} \textbf{247} (1984), 83--103.

\bibitem[LL]{LL}
J. Lepowsky and H. Li, \textit{Introduction to Vertex
Operator Algebras and Their Representations}, Progress in Math.,
Vol. 227, Birkh\"auser, Boston, 2003.

\bibitem[Li]{Li2}
H. Li, Determining fusion rules by $A(V)$-modules and bimodules, \textit{J.
Algebra}  \textbf{212}  (1999),  515--556.

\bibitem[MY]{MY}
R. McRae and J. Yang, Vertex algebraic intertwining operators among generalized Verma modules for $\widehat{\mathfrak{sl}(2,\C)}$, \textit{Trans. Amer. Math. Soc.} \textbf{370} (2018), 2351--2390. 

\bibitem[Mi]{Mi}
M. Miyamoto, $C_1$-cofiniteness and fusion products of vertex operator algebras, \textit{Conformal
Field Theories and Tensor Categories, Proceedings of a Workshop Held at Beijing International
Center for Mathematics Research}, ed. C. Bai, J. Fuchs, Y.-Z. Huang, L. Kong, I. Runkel and C.
Schweigert, Mathematical Lectures from Beijing University, Vol. 2, Springer, New York, 2014,
271--279.

\bibitem[Ri]{Ri}
D. Ridout, $\widehat{\mathfrak{sl}}(2)_{-1/2}$: a case study, \textit{Nuclear Phys. B} \textbf{814} (2009), 485--521.



\bibitem[Ru]{Ru}
I. Runkel, A braided monoidal category for free super-bosons, \textit{J. Math. Phys.} \textbf{55} (2014), 041702, 59 pp.

\end{thebibliography}
\end{document}